\documentclass{amsart} 
\usepackage{epsf} 
\usepackage{amssymb,latexsym, amsmath, amscd, array, graphicx}

\swapnumbers 
\numberwithin{equation}{section} 

\theoremstyle{plain} 
\newtheorem{thm}{Theorem}[section]

\theoremstyle{definition} 
\newtheorem{defin}[thm]{Definition}

\def\vb{\protect\operatorname{vb}}

\def\proj{\protect\operatorname{proj}}
\def\cl{\protect\operatorname{cl}}
\def\cr{\protect\operatorname{cr}}
\def\b{\protect\operatorname{b}}

\def\R{{\mathbb R}}

\def\1{\hbox{\rm\rlap {1}\hskip.03in{\rom I}}} 
\def\Bbbone{{\rm1\mathchoice{\kern-0.25em}{\kern-0.25em} 
{\kern-0.2em}{\kern-0.2em}I}} 

\def\wt{\widetilde}

\begin{document}
\title[Some Corollaries of Manturov's projection Theorem] 
{Some Corollaries of Manturov's Projection Theorem} 
\author[V.~Chernov]{Vladimir V. 
Chernov}
\address{V. Chernov, Department of Mathematics, 
6188 Kemeny Hall, Dartmouth College, Hanover NH 03755, 
USA} 
\email{Vladimir.Chernov@dartmouth.edu} 

\begin{abstract}
In our works with Stoimenow, Vdovina and with Byberi, we introduced the virtual canonical genus $g_{vc}(K)$ and the virtual bridge number $\vb(K)$ invariants of virtual knots. One can see from  the definitions that for a classical knot $K$ the values of these invariants are less or equal than the classical canonical genus $g_c(K)$ and the bridge number $b(K)$ respectively. We use Manturov's projection from the category of virtual knot diagrams to the category of classical knot diagrams, to show that for every classical knot type $K$ we have $g_{vc}(K)=g_c(K)$ and  $\vb(K)=b(K)$.
\end{abstract}


\maketitle

\section{Introduction}
We work in the $C^{\infty}$ category and the word ``smooth'' means $C^{\infty}.$

The Virtual Knot Theory was introduced by L.~Kauffman~\cite{Kauffman}, one of the possible uptodate references is~\cite{ManturovandIlyutko}. Let us recall some of the basic notions of virtual knot theory.

A knot is a smooth embedding $S^1\to \R^3$. A knot can be described by a knot diagram which 
is a generic immersion $S^1\to \R^2$ equipped with information
about under-passes and over-passes at double points. The sign of a double point $d$ is $+$ if the $3$-frame consisting of  the velocity vector of the over-passing branch, 
the velocity vector of the under-passing branch and the vector from the preimage of $d$ on the under-passing branch to the preimage of $d$ on the over-passing branch gives the positive orientation of $\R^3.$ Otherwise the sign of $d$ is $-.$

A knot diagram $D$ gives rise to 
a Gauss diagram $G_D$ which is a circle parameterizing the knot with each pair of preimages of double points 
of $D$ connected by an oriented signed chord. The chords are oriented from the preimage of the double point on the 
over-passing branch to the preimage of the double point on the under-passing branch.
The sign of a chord is the sign of the corresponding double point. The resulting Gauss diagram $G_D$ 
is {\em the Gauss diagram of the knot diagram $D$.\/}

Gauss diagrams that are obtainable as Gauss diagrams of some knot diagrams
are called {\em realizable.\/} 
A knot diagram corresponding to a realizable Gauss diagram can be recovered
only up to a certain ambiguity. However the isotopy type of the corresponding
knot is recoverable in a unique way, see Goussarov-Polyak-Viro~\cite{GPV}.

It is well-known that two knots described by their knot diagrams are ambient isotopic
if and only if one can change one diagram to the other by a sequence of ambient isotopies of the diagram and 
of {\em Reidemeister moves.\/}
Reidemeister moves can be easily encoded in the language of Gauss diagrams, see for example~\cite{GPV}. The equivalence classes of the, not necessarily realizable, 
Gauss diagrams modulo the resulting operations on Gauss diagrams are called {\it virtual knots.} 

We say that a virtual knot $\wt K$ is {\it realizable\/} by a classical knot $K$ if for a planar diagram  $D$
of $K$ we have that the Gauss diagram $G_D$ represents the virtual knot $\wt K.$ (So in this paper the word realizable could refer to both virtual knots and Gauss diagrams. The meaning of the word realizable should be clear from the context.)

Every Gauss diagram $G$ can be realized by a virtual knot diagram $\wt D$ which is similar to the classical knot diagram, but some of the crossings are virtual. Virtual crossings are marked by small circles and the notion of the under-passing and overpassing branch at the virtual crossing is not defined. Such virtual knot diagram $\wt D$ is not unique, but every $\wt D$ gives rise to a unique 
Gauss diagram $G_{\wt D}$ which is constructed similarly to $G_D$, but all the virtual crossings are ignored and do not give rise to any chords.  
One can define virtual Reidemeister moves for virtual knot diagrams so that the equivalence 
classes of virtual knot diagrams modulo these moves are the virtual knots,~see~\cite{Kauffman}.

The results of our work use the following Theorem of Manturov~\cite[Theorem 4]{ManturovNew} that proves the existence of a projection from the category of virtual knots to the category of 
classical knots.

\begin{thm}[Manturov]\label{Manturov}
There is a well defined map $\proj$ from the set of all Gauss diagrams to the set of realizable Gauss diagrams such that:
\begin{enumerate}
\item If two Gauss diagrams $G_1$ and $G_2$ are equivalent, then so are the realizable Gauss diagrams $\proj(G_1)$ and $\proj(G_2).$
\item The Gauss diagram $\proj (G)$ is obtained from $G$ by removing some chords.
\item If the Gauss diagram $G$ is realizable, then $\proj(G)=G.$
\end{enumerate}

\end{thm}

The classical crossing number $cl(\wt K)$ of a virtual knot $\wt K$ is the minimal number of (non-virtual) crossings in a virtual diagram $\wt D$ of $\wt K.$ The crossing number $cr(K)$ of a classical knot $K$ is the minimal number of crossings in a (non-virtual) planar diagram of $K$. Clearly if a virtual knot $\wt K$ is realizable by a classical knot type $K$, then 
$\cl(\wt K)\leq \cr(K)$. Using Theorem~\ref{Manturov} Manturov~\cite[Corollary 1]{ManturovNew} proved that $\cl(\wt K)=\cr(K).$ 

We use Manturov's projection Theorem~\ref{Manturov} to prove equalities between the virtual bridge number and bridge number of a classical knot, and of the virtual canonical genus and the canonical genus of a classical knot.

\section{Equality of the virtual canonical genus and of the canonical genus invariants for classical knots}

Let $G$ be a Gauss diagram. Substitute the core circle of $G$ by an annulus whose core is the circle, and substitute the chords by non-intersecting thing strips attached to the annulus in such a way that the resulting surface with boundary is orientable. We denote the resulting {\it ribbon surface} by $F'_G.$ 
Glue the connected components of the boundary of $F'_G$ with disks to get a closed oriented surface $F_G$.  Note that this surface does not depend on the number of twists in the glued strips as soon as the result is orientable. One can show~\cite[Theorem 2.5]{StoimenowTchernovVdovina} that if $G=G_D$ is the Gauss diagram of the classical knot diagram $D$, then the genus of $F_G$ is equal to the genus of the canonical  Seifert surface obtained by applying the Seifert algorithm to $D,$ see~\cite[page 120]{Rolfsen}.

Recall that for a knot type $K$, the {\it canonical genus\/} $g_c(K)$ of $K$ is the minimum genus of the canonical Seifert surface of a knot diagram $D$ of $K$. The minimum is taken over all diagrams $D$ of $K.$ The {\it genus\/} of $K$ is the minimal genus of a Seifert surface (that does not have to come from a Seifert algorithm) for a knot realizing $K.$
Moriah~\cite{Moriah} showed that the difference between $g(K)$ and $g_{vc}(K)$ can be arbitrarily large, that is for every $n\in N$ there exists a classical knot $K$ such that $g_c(K)-g(K)>n.$
Results in the similar spirit were later obtained by Kawauchi~\cite{Kawauchi}, Stoimenow~\cite{Stroimenow}, see also Kobayashi and Kobayashi~\cite{KobayashiKobayashi}.

\begin{defin} 
Following~\cite{StoimenowTchernovVdovina} we define the canononical virtual genus $g_{vc}(\wt K)$ of the virtual knot type $\wt K$ to be the minimum of the genera of $F_G$ taken over all the Gauss diagram $G$ realizing $\widetilde K.$
\end{defin}

From the definition of $g_{vc}(\wt K)$ and~\cite[Theorem 2.5]{StoimenowTchernovVdovina} it is clear that if $\wt K$ is realizable by a classical knot $K$, then 
$g_{vc}(\wt K)\leq g_c(K),$ see~\cite[Corollary 2.8]{StoimenowTchernovVdovina}.

We prove the following Theorem.

\begin{thm}\label{canonicalgenus}
 If a virtual knot $\wt K$ is realizable by a classical knot $K$, then $g_{vc}(\wt K)=g_c(K).$
\end{thm}

\begin{proof} Let $\wt D$ be a virtual diagram for $\wt K$ such that the genus of $F_{G_{\wt D}}$ is equal to $g_{vc}(\wt K).$  By Theorem~\ref{Manturov} the realizable Gauss diagram $\proj(G_{\wt D})$ represents the classical knot $K$. So the genus of $F_{\proj(G_{\wt D})}$ equals to the genus of the canonical Seifert surface for the knot diagram associated to the realizable Gauss diagram  $\proj(G_{\wt D})$ of $K.$ Thus the genus of $F_{\proj(G_{\wt D})}$ is at least $g_c(K).$ Since $g_{vc}(\wt K)\leq g_c(K)$, to prove the Theorem it suffices to show that the genus of $F_{\proj(G_{\wt D})}$ is less or equal to the genus of $F_{G_{\wt D}}$.

Assume that a Gauss diagram $G_1$ is obtained from $G_2$ by deleting one chord. Let us compare the genus of $F_{G_1}$ and of $F_{G_2}.$
The ribbon surface $F'_{G_1}$ will have one less strip than $F'_{G_2}$, so $\chi(F'_{G_1})=\chi(F'_{G_2})+1$, where $\chi$ is the Euler characteristic.  

If the two edges of the strip in $F'_{G_2}$ corresponding to the deleted chord of $G_2$ belong to two components of the boundary of $F'_{G_2}$, then these two boundary components of $F'_{G_2}$ will merge  into one boundary component for $F'_{G_1}$. Thus the number of disks we have to glue to $F'_{G_1}$ to get $F_{G_1}$ is one less than the number of disks we have to glue to $F'_{G_2}$ to get $F_{G_2}.$ Thus $\chi(F_{G_1})=\chi(F_{G_2})$ and the closed oriented surfaces $F_{G_1}$ and $F_{G_2}$ have the same genus.

If the two edges of the strip in $F'_{G_2}$ corresponding to the deleted chord of $G_2$ belong to the same component of the boundary of $F'_{G_2}$, then this boundary component of $F'_{G_2}$ will disintegrate into two boundary component for $F'_{G_1}$. Thus the number of disks we have to glue to $F'_{G_1}$ to get $F_{G_1}$ is one more than the number of disks we have to glue to $F'_{G_2}$ to get $F_{G_2}.$ Thus $\chi(F_{G_1})=\chi(F_{G_2})+2$ and the closed oriented surfaces $F_{G_1}$ has genus one less than $F_{G_2}$.

In both possible cases the genus of $F_{G_1}$ is less than the genus of $F_{G_2}$. By Manturov's Theorem~\ref{Manturov} 
$\proj(G_{\wt D})$ is obtained from $G_{\wt D}$ by deleting certain chords. So the genus of $F_{\proj(G_{\wt D})}$ is less or equal to the genus of $F_{G_{\wt D}}$.
\end{proof}

\section{Equality of the virtual bridge number and of the bridge number invariants for classical knots}
Let us recall the definition of the virtual bridge number following our work with Byberi~\cite{ByberiChernov}.

A {\it bridge\/} in a classical knot diagram $D$ is an arc between two consecutive under-passes that contains nonzero 
many over-passes. The {\it bridge number\/} $\b(K)$ of a classical knot $K$, 
is the minimum number of bridges in a knot diagram $D$ realizing $K.$

Given a Gauss diagram $G_D$, the number of bridges in $D$ can be counted as the 
number of circle arcs between two consecutive  arrow heads that contain nonzero many arrow tails. Note that 
this number is well defined even for the Gauss diagram $G_{\wt D}$ of a virtual knot diagram $\wt D,$ and we call it the number of bridges in a 
Gauss diagram $G_{\wt D}$. 

\begin{defin}
The {\em virtual bridge number\/} $\vb (\wt K)$ of a virtual knot $\wt K$ is the 
minimum number of bridges in $G_{\wt D}$ over all the Gauss diagrams $\wt D$ realizing $\wt K.$ 
\end{defin}

If $\vb(\wt K)=0,$ then $\wt K$ is the unknot. There are no classical knots $K$ with $\b(K)=1$, but we showed that there are infinitely many different virtual knots $\wt K$ 
with $\vb(\wt K)=1$, see~\cite[Theorem 2.5]{ByberiChernov}.

Clearly if a virtual knot $\wt K$ is realizable by a classical knot $K$, then $\vb(\wt K)\leq \b(K),$ see~\cite[page 1]{ByberiChernov}. In~\cite{ByberiChernov} we said that it is plausible that for many virtual knots realizable by classical knots, we have $\vb(\wt K)=\b(K).$  The following theorem says that this is always true.

\begin{thm}\label{bridgenumber}
 If a virtual knot $\wt K$ is realizable by a classical knot $K$, then $\vb(\wt K)=\b(K).$
\end{thm}

\begin{proof}
Take a virtual knot diagram $\wt D$ realizing $\wt K$ such that the number of bridges in $\wt D$ equals to $\vb(\wt K).$ Since $\vb(\wt K)\leq \b(K)$ to prove the Theorem it suffices to show that $\b(K)\leq \vb(\wt K).$

Assume that a Gauss diagram $G_1$ is obtained from $G_2$ by deleting one chord. This chord could have participated in the count of the number of bridges in $G_2$ in one of the two possible ways. Either its arrow head was one of the bridge ends or its arrow tail was located between two consecutive arrow heads. (There is also an option that it was not influencing the number of bridges, for example if it did not intersect any other chords.) In all of these cases the number of bridges in $G_1$ is less or equal than the number of bridges in $G_2.$ 

By Manturov's Theorem the realizable Gauss diagram $\proj(G_{\wt D})$ represents the classical knot $K$ and $\proj(G_{\wt D})$ is obtained from $G_{\wt D}$ by deleting a certain number of chords. So the number of bridges of $\proj(G_{\wt D})$ is less or equal to the number of bridges of $G_{\wt D}$. It is less or equal than $\vb (\wt K)$ and greater or equal than $b(K).$ Thus $\vb(\wt K)\geq b(K)$.
\end{proof}

{\bf Acknowledgments.} 
The author is thankful to Alina Vdovina for the mathematical discussions 
during which the notion of the virtual bridge number was introduced. He is also grateful to Vassily Manturov and Denis Ilyutko for many useful discussions.

This work was partially supported by a grant from the Simons Foundation ($\#235674$ to Vladimir Chernov).


\begin{thebibliography}{10}
\bibitem{ByberiChernov}
E.~Byberi, V.~Chernov: {\it Virtual bridge number one knots.\/} 
Commun.~Contemp.~Math. {\bf 10} (2008), suppl.~1, 1013-1021.

\bibitem{GPV}
M.~Goussarov, M.~Polyak and O.~Viro:
{\em Finite-type invariants of classical and virtual knots,\/}
Topology 39 (2000), no. 5, 1045--1068.

\bibitem{Kawauchi}
A.~Kawauchi: {\it On coefficient polynomials of the skein polynomial of an oriented link,\/} 
Kobe J. Math. {\bf 11} (1) (1994), 49-68.

\bibitem{Kauffman}
L.~H.~Kauffman:
{\em Virtual Knot Theory, \/} European J.~Comb. 20 (1999), 663-690.

\bibitem{KobayashiKobayashi}
M.~Kobayashi and T.~Kobayashi: {\it On canonical genus and free genus of knot,\/} J. Knot
Theory Ramifications {\bf 5} (1) (1996), 77-85.


\bibitem{ManturovNew}
V.~O.~Manturov: {\it Parity and Projection from Virtual Knots to Classical Knots,\/} preprint (2012), 19 pages


\bibitem{ManturovandIlyutko}
V.~O.~Manturov and D.~P.~Ilyutko: {\it Virtual knots. The state of the art.\/} On Knots and Everything {\bf 51}, 
Singapore (2012)



\bibitem{Moriah}
Y.~Moriah: {\it On the free genus of knots,\/} Proc. Amer. Math. Soc. {\bf 99} (2) (1987), 373-379.

\bibitem{Rolfsen}
D.~Rolfsen: {\it Knots and Links,\/} Math.~Lecture Ser.~$7$, Publish or Perish, Berkeley, CA (1976)

\bibitem{Stroimenow}
A.~Stoimenow: {\it Jones polynomial, genus and weak genus of a knot,\/} Ann.~Fac.~Sci.~Toulouse
{\bf 8} (4) (1999), 677-693.

\bibitem{StoimenowTchernovVdovina}
A.~Stoimenow, V.~Tchernov, A.~Vdovina:
{\it The canonical genus of a classical and virtual knot,\/}
Proceedings of the Conference on Geometric and Combinatorial Group Theory, Part II (Haifa, 2000). 
Geom.~Dedicata {\bf 95} (2002), 215-225.

\end{thebibliography}
\end{document}